\documentclass[a4paper]{amsart}
\usepackage[latin9]{inputenc}
\usepackage{a4}
\usepackage{amsfonts}
\usepackage{amssymb}
\usepackage{amsmath}
\usepackage{amsthm}
\usepackage{changepage}
\usepackage{nomencl}
\usepackage{geometry}
\usepackage{tikz}
\usepackage{verbatim}
\usetikzlibrary{matrix}
\usepackage{hyperref}
\usepackage[toc,page]{appendix}
\usepackage{mathrsfs}
\usepackage[square, numbers, comma]{natbib}

\usepackage{enumerate}
\usepackage{amsbsy}

\usepackage[all]{xy}
\usepackage{pb-diagram,pb-xy}
   \bibpunct{[}{]}{,}{n}{}{,}
\input cyracc.def

\begin{document}
\providecommand{\keywords}[1]{\textbf{\textit{Keywords: }} #1}
\newtheorem{theorem}{Theorem}[section]
\newtheorem{lemma}[theorem]{Lemma}
\newtheorem{proposition}[theorem]{Proposition}
\newtheorem{corollary}[theorem]{Corollary}
\newtheorem{problem}[theorem]{Problem}
\newtheorem{question}[theorem]{Question}
\newtheorem{conjecture}[theorem]{Conjecture}
\newtheorem{claim}[theorem]{Claim}
\newtheorem{condition}[theorem]{Condition}

\theoremstyle{definition}
\newtheorem{definition}[theorem]{Definition} 
\theoremstyle{remark}
\newtheorem{remark}[theorem]{Remark}
\newtheorem{example}[theorem]{Example}
\newtheorem{condenum}{Condition}

\def\p{\mathfrak{p}}
\def\q{\mathfrak{q}}
\def\s{\mathfrak{S}}
\def\Gal{\mathrm{Gal}}
\def\Ker{\mathrm{Ker}}
\def\soc{\mathrm{soc}}
\def\Coker{\mathrm{Coker}}
\newcommand{\cc}{{\mathbb{C}}}   
\newcommand{\ff}{{\mathbb{F}}}  
\newcommand{\nn}{{\mathbb{N}}}   
\newcommand{\qq}{{\mathbb{Q}}}  
\newcommand{\rr}{{\mathbb{R}}}   
\newcommand{\zz}{{\mathbb{Z}}}  
\def\K{\kappa}

\title{Unramified extensions of quadratic number fields with Galois group $2.A_n$}
\author{Joachim K\"onig}
\address{Department of Mathematics Education, Korea National University of Education, Cheongju 28173, South Korea}
\email{jkoenig@knue.ac.kr}
\footnotetext{{\ 2020 Mathematics Subject Classification.} Primary 12F12; Secondary 11R21, 11R29.} 
\begin{abstract}
We realize infinitely many covering groups $2.A_n$ (where $A_n$ is the alternating group) as the Galois group of everywhere unramified Galois extensions over infinitely many quadratic number fields. After several predecessor works investigating special cases or proving conditional results in this direction, these are the first unramified realizations of infinitely many of these groups.\end{abstract}
\keywords{Inverse Galois theory; unramified extensions; embedding problems.}
\maketitle

\section{Statement of main result}

It is an open problem in inverse Galois theory whether every finite group occurs as the Galois group of some unramified Galois extension of a quadratic number field.
Here, the expression ``unramified extension" is to be understood as including the archimedean primes. The Cohen-Lenstra heuristics and their non-abelian generalizations (\cite{Wood}) provide expectations for the asymptotic count of such quadratic fields (when the finite group has been fixed). Even the question whether there are infinitely many such quadratic fields for a given group $G$ has been answered only for very particular groups $G$, among them the alternating groups $A_n$, see, e.g., \cite{Uchida}. See also the predecessor papers \cite{KoeSL27} and \cite{KoeSL25} for more background on this problem. In this note, we investigate the embeddability of unramified $A_n$ extensions in unramified $2.A_n$ extensions (where $2.A_n$ denotes the unique degree-$2$ stem cover of $A_n$). Concretely, we show the following.

\begin{theorem}
\label{thm:main}
Let $n\ge 4$ be such that either $n\equiv 3 \bmod 8$ is a prime, or $n\equiv 2\bmod 8$ and $n-1$ is a prime or a square. 
Then there exist infinitely many quadratic number fields which possess an unramified Galois extension with Galois group $2.A_n$, the double covering group of $A_n$.
\end{theorem}

 Note that the assumptions hold for infinitely many values of $n$ (in particular, for all $n=(2k+1)^2+1$, $k\in \mathbb{N}$, as well as for a positive density subset of all prime numbers $n$ by Dirichlet's prime number theorem). 
In \cite{KoeSL27}, the author gave a \underline{{\it wrong}} (cf.\ \cite{Koe_SL27_Err}) proof of the analogous assertion for $n=7$, based on a misreading of the group structure of a double covering group of $S_7$. 
The underlying Galois-theoretic ideas were however not without merit and could have been put to much better use, which we will demonstrate here in the spirit of ``better late than never". 
See also \cite{KimK}, which showed the same result for {\it all} $n$, but conditional on the (wide-open) Bunyakovsky conjecture on prime values of irreducible polynomials. Comparable embedding problems in the context of unramified extensions have been investigated for other groups, for example for the quaternion group in \cite{Lemmer}.

In the last section, we present an alternative approach, which, in spite of being slightly computation-heavy and at this point only yielding the analogous result for one further open case ($n=6$), should be of interest in its own right.

\section{Technical tools}
In the following, for $n\ge 4$, denote by $2.S_n^+$ the unique stem double cover of $S_n$ in which the transpositions lift to elements of order $2$. This contains the double cover $2.A_n$ of $A_n$ as an index-$2$ normal subgroup; cf.\ \cite[Section 2.7.2]{Wilson}.
The proof of Theorem \ref{thm:main} will make use of certain well-known results on solutions of central embedding problems, related to the group $2.S_n^+$. For general background on such embedding problems, see, e.g., \cite[Chapter IV.1.2]{MM} or \cite[Section 2.1]{Serre}.
\begin{proposition}
\label{prop:embed}
Assume that $K/\mathbb{Q}$ is a Galois extension with Galois group $S_n$ ($n\ge 4$) and $F/\mathbb{Q}$ is the unique quadratic subextension of $K/\mathbb{Q}$. Assume furthermore that for each prime number $p$ ramified in $K/\mathbb{Q}$, one of the following holds:
\begin{itemize}
\item[i)] The inertia group at (any prime extending) $p$ is generated by a transposition, say $(a,b)\in S_n$, and the decomposition group at $p$ is contained in $\langle (a,b)\rangle\times \textrm{Alt}(\{1,\dots, n\}\setminus \{a,b\})$.
\item[ii)] The inertia group at $p$ is generated by an involution with $4j+1$ ($j\ge 0$) disjoint transpositions, and every involution in the decomposition group at $p$ has $d\equiv 0,1 \bmod 4$ disjoint transpositions.\footnote{In particular, this is fulfilled if the residue degree at $p$ is odd, since there is then a unique involution in the decomposition group.}
\end{itemize}
 Then $K/\mathbb{Q}$ embeds into a $2.S_n^+$-extension $L/\mathbb{Q}$ in which all the ramified non-archimedean primes have inertia group of order $2$ lying outside of $2.A_n$. In particular,
the extension $L/F$ is a $2.A_n$-extension unramified at all non-archimedean primes. 
\end{proposition}
\begin{proof}
The second assertion is directly implied by the first due to the multiplicativity of ramification indices in towers of Galois extensions. 
The first assertion 
is the combination of the following two assertions: a) the central embedding problem given by 
\begin{equation}\label{eq:embed} 1\to C_2\to 2.S_n^+\to \textrm{Gal}(K/\mathbb{Q}) \cong S_n\to 1\end{equation}
 possesses a proper solution; and b), the solution field $L$ can be chosen such that no non-archimedean prime ramifies in $L/K$ (and thus indeed in $L/F$, since assumptions i) and ii) imply that $K/F$ is unramified at all non-archimedean primes).
Both assertions can be tackled using well-known local-global principles: notably, by \cite[Cor.\ IV.10.2 and Remark]{MM}, the embedding problem possesses a solution as soon as for every non-archimedean prime $p$, the associated local embedding problem possesses a solution. Note furthermore that the solution to the global embedding problem will automatically be proper, since the extension $2.S_n^+$ of $S_n$ is non-split. Moreover, \cite[Prop.\ 2.1.7]{Serre} ensures that, in case of solvability of the embedding problem, a solution field $L$ maybe chosen such that the set of ramified non-archimedean primes of $L/\mathbb{Q}$ equals the one of $K/\mathbb{Q}$ and restriction to inertia (for each of these ramified primes) can be chosen to equal that of any given local solution.

In total, we are therefore left with showing the following: For any prime number $p$ ramified in $K/\mathbb{Q}$, the associated local embedding problem has a solution whose inertia group remains of order $2$ (i.e., the ramification index of $p$ in $K/\mathbb{Q}$).

Let $D$ be the decomposition group in $K/\mathbb{Q}$ at a ramified prime $p$. In case i), we have $D = \langle(a,b)\rangle\times U$ for a cyclic subgroup $U$ of $Alt(\{1,\dots, n\}\setminus\{a,b\})$. We claim that the preimage of $D$ in $2.S_n^+$ is abelian. This (together with the fact that the inertia group $\langle(a,b)\rangle$ lifts to an order $2$ group in $2.S_n^+$) then implies that there is no local obstruction at $p$ to the solvability of the embedding problem given by \eqref{eq:embed}, since the degree-$2$ unramified extension of the completion of $K/\mathbb{Q}$ at $p$ is a solution of the induced local embedding problem. To see the claim, recall that it is well-known (e.g., \cite[Section 2.7.2]{Wilson}) that in $G=2.S_n^+$, the commutator $[x,y]:=x^{-1}y^{-1}xy$ of two (lifts of) transpositions $x,y$ with disjoint support is the central involution of $G$. Combine this with the commutator identity $[x,zy] = [x,y]\cdot [x,z]^y$ to see instantly that the commutator of a lift of the transposition $(a,b)$ with any lift of an element of $Sym(\{1,\dots, n\}\setminus\{a,b\})$ is trivial if and only if the latter element is an even permutation.

In case ii), since the inertia group at $p$ in $K/\mathbb{Q}$ is central in the decomposition group and has cyclic quotient, the whole decomposition group $D$ is abelian. Write it as $D = U\times V$ where $U$ is the $2$-Sylow group. The elements of order $\le 2$ in $U$ form a subgroup $W$ isomorphic to $C_2$ or $C_2\times C_2$. Since the involutions of $S_n$ which lift to elements of order $2$ in $2.S_n^+$ are exactly those with $d\equiv 0,1 \bmod 4$ transpositions, $W$ has a preimage in $2.S_n^+$ in which every element has order $\le 2$, i.e., is elementary-abelian. There is thus a subgroup of $2.S_n^+$ mapping isomorphically onto $W$, thus also one mapping isomorphically onto $U$, and therefore finally also one mapping isomorphically onto $D$. But then the associated local embedding problem at $p$ has a trivial solution.
\end{proof}

Compare also the proofs of \cite[Lemma 3.1 and Lemma 4.1]{KoeSL25}, which deal with the special cases $n=5$ and $n=6$ of Proposition \ref{prop:embed}.  

In view of Proposition \ref{prop:embed}, the bulk of the proof of Theorem \ref{thm:main} is reduced to finding $S_n$ extensions of $\mathbb{Q}$ with certain restrictions on the inertia and decomposition groups at ramified primes. The main idea now, attempted also in \cite{KoeSL27}, is to enforce such conditions ``on the geometric level", i.e., construct function field extensions of $\mathbb{Q}(t)$ with Galois group $S_n$, and whose ramified places have transposition inertia and decomposition group embedding into $S_2\times A_{n-2}$ as in Proposition \ref{prop:embed}; and to then use known results on the local behavior of specializations of function field extensions to obtain these local conditions for Galois extensions of $\mathbb{Q}$.

\begin{theorem}
\label{thm:tech1}
Let $n\ge 7$ be an integer with $n\equiv 3 \bmod 4$. 
Let $t\in \mathbb{Z}$ be coprime  to $n(n-1)$. Let 
$f(X):=f_{t}(X)=X^{n-1}(X-1) +n^{\frac{n-1}{2}}\cdot (\frac{n-1}{2})^{-n}\cdot t^{n-1}
\in \mathbb{Q}[X]$. 
Then, for 
infinitely many such 
 $t$, the splitting field $K$ of $f$ is an $S_n$ extension of $\mathbb{Q}$ with the following properties.
\begin{itemize}
\item[i)] $K$ is unramified at all prime divisors of $n-1$.
\item[ii)] For all prime numbers which are ramified in $K$ and do not divide $n$, the inertia group is generated by a transposition, say $(a,b)\in S_n$, and the decomposition group is contained in $Sym(\{a,b\})\times Alt(\{1,\dots, n\}\setminus\{a,b\})$.
\item[iii)] If moreover $n=p$ is a prime number, then additionally the inertia group at $p$ is generated by an involution with exactly one fixed point, and the residue degree at $p$ is odd.
\end{itemize}

\end{theorem}

\begin{proof}
To see i), write $(\frac{n-1}{2})^n f((\frac{n-1}{2})^{-1}X) = X^{n-1}(X-\frac{n-1}{2})+n^{\frac{n-1}{2}}\cdot t^{n-1} 
\equiv X^n+t^{n-1} 
\bmod n-1$. Since $t$ is coprime to $n-1$, the latter polynomial is separable modulo any prime divisor of $n-1$, implying that these primes are unramified in $K$.

To see ii), note that all prime numbers other than the divisors of $n(n-1)$ are not fixed divisors of the discriminant of $f_t(X)$ as $t$ varies over $\mathbb{Z}$. In particular, by choosing $t$ in a suitable residue class, we may assume any prescribed finite set $\mathcal{S}_0$ to be unramified in $K$. All further ramified primes may therefore be assumed ``non-exceptional", in the sense of  \cite[Theorem 4.1]{KLN},
 for the extension of the rational function field $\mathbb{Q}(T)$ given by the splitting field of $f_T(X)$ (i.e., the polynomial arising from $f_t$ via replacing $t$ by the transcendental $T$) .
 Now write $S=n^{\frac{n-1}{2}}\cdot (\frac{n-1}{2})^{-n}\cdot T^{n-1}$. The function field extension $E/\mathbb{Q}(S)$ given by $X^{n-1}(X-1)+S$ is well-known to be a $\mathbb{Q}$-regular $S_n$-extension ramified exactly at the three places $S\mapsto 0$, $S\mapsto \infty$ and $S\mapsto \frac{(n-1)^{n-1}}{n^n}$, with corresponding inertia groups generated by an $n-1$-cycle, an $n$-cycle and a transposition respectively (cf., e.g., \cite[p.42]{Serre}).
Theorem 4.1 of \cite{KLN} therefore implies that all further ramified non-archimedean primes $q$ in $K/\mathbb{Q}$ ``inherit" (in a sense to be made precise below) their local behavior from one of the ramified places of $E/\mathbb{Q}(S)$. On the other hand, since the specialization value $S\mapsto n^{\frac{n-1}{2}}\cdot (\frac{n-1}{2})^{-n}\cdot t^{n-1}\in \mathbb{Q}$ is $q$-integral and of $q$-valuation divisible by $n-1$, the so-called Specialization Inertia Theorem (e.g.,\cite{Beck} or \cite{Leg}) implies there is no ramified prime inheriting its local behavior from $S\mapsto \infty$ or $S\mapsto 0$ (the latter since the specialization value and the branch point $0$ meet at $q$ with multiplicity divisible by the ramification index at $0$).

 The only primes left to consider are  the divisors (coprime to $n(n-1)$) of the numerator of 
 $$n^{\frac{n-1}{2}}\cdot (\frac{n-1}{2})^{-n}\cdot t^{n-1} - \frac{(n-1)^{n-1}}{n^n} = (\frac{n(n-1)}{2})^n\cdot (n^{\frac{3n-1}{2}} t^{n-1} - (\frac{n-1}{2})^{2n-1}2^{n-1}).$$ 
 More precisely, \cite[Theorem 4.1]{KLN} implies that for all prime divisors $q$ of the latter factor of odd multiplicity,\footnote{Prime divisors of even multiplicity would once again be unramified, by the specialization inertia theorem, since this implies meeting of the specialization value and the branch point $\frac{(n-1)^{n-1}}{n^n}$ of multiplicity divisible by the ramification index $2$ at this branch point.} the inertia group is generated by a transposition $(a,b)$, and the decomposition group is of the form $\langle(a,b)\rangle\times U_q$, where $U_q\le Sym(\{1,\dots, n\}\setminus\{a,b\})$ is the cyclic group generated by the Frobenius at $q$ in the residue extension $F/\mathbb{Q}$ at $S\mapsto \frac{(n-1)^{n-1}}{n^n}$ of $E/\mathbb{Q}(S)$. It remains to show that this Frobenius is an even permutation of $S_{n-2}$. To see this, note that the discriminant of $F$ equals, up to squares, the discriminant of $g(X):=\sum_{i=0}^{n-2} (i+1)X^i$, since the latter polynomial is the degree-$n-2$ factor of $f_{(n-1)^{n-1}/n^n}(\frac{n-1}{n}X)$. This discriminant is $(-1)^{\frac{n+1}{2}} \cdot 2\cdot n^{n-3} (n-1)^{n-4} = 2\cdot n^{n-3} (n-1)^{n-4}$ (for $n\equiv 3 \bmod 4$), so that it now remains to show that any prime divisor $q$ as above splits in $\mathbb{Q}(\sqrt{\frac{n-1}{2}})$. This requirement is, in fact, the reason for the choice of coefficient $\gamma:= n^{\frac{n-1}{2}}\cdot (\frac{n-1}{2})^{-n}$ ahead of $t^{n-1}$ in the definition of $f_t(X)$. Namely, we may choose a transcendental $U$ with $S=\gamma U^2 = \gamma (T^{\frac{n-1}{2}})^2$, so that $\mathbb{Q}(U)$ becomes a quadratic subextension of the extension $\mathbb{Q}(T)/\mathbb{Q}(S)$. We can use this to interpret $f_t(X)$ as the specialization of 
 $E(U)/\mathbb{Q}(U)$ at $U\mapsto t^{\frac{n-1}{2}}$, and repeat the above to see that any ramified prime $q$ of $K/\mathbb{Q}$ coprime to $n(n-1)$ must inherit its local behavior from a place $U\mapsto u_0 \in \overline{\mathbb{Q}\cup\{\infty\}}$ of ramification index $2$ in  $E(U)/\mathbb{Q}(U)$. But being a preimage of $S\mapsto \frac{(n-1)^{n-1}}{n^n}$ under the quadratic pullback $S=\gamma U^2$, such a place has residue field $\mathbb{Q}(\sqrt{\gamma \frac{(n-1)^{n-1}}{n^n}}) = \mathbb{Q}(\sqrt{\gamma n}) = \mathbb{Q}(\sqrt{\frac{n-1}{2}})$, so that a necessary requirement for the mod-$q$ meeting of the ($\mathbb{Q}$-rational!) specialization value $U\mapsto t^{\frac{n-1}{2}}$ with such a place is the $\mathbb{F}_q$-rationality of the latter, meaning that $q$ is split in $\mathbb{Q}(\sqrt{\frac{n-1}{2}})$, as claimed.
 
 To see iii), assume additionally that $n=p$ is a prime number. Note first that the Newton polygon of $f(X)$ shows instantly that only one of the roots of $f$ lies in $\mathbb{Q}_p$ whereas all others have half-integral $p$-adic valuation and therefore lie in a ramified extension of $\mathbb{Q}_p$. On the other hand, choose the quadratic ramified extension $\mathbb{Q}_p(\sqrt{-2p})/\mathbb{Q}_p$.
We have $t^{1-p}p^{-\frac{p-1}{2}} \cdot f_t(\sqrt{-2 p} \cdot tX) = (-2)^{\frac{p-1}{2}}X^{p-1}(\sqrt{-2 p} \cdot tX - 1) + (\frac{p-1}{2})^{-p}$, which modulo the maximal ideal of $\mathbb{Q}_p(\sqrt{-2p})$ reduces, up to constant factor, to the separable polynomial 
$X^{p-1}+ (-1)^{\frac{p+1}{2}} (p-1)^{-p}\cdot 2^{p-\frac{p-1}{2}}\equiv X^{p-1} - 2^{\frac{p+1}{2}} = (X^{\frac{p-1}{2}}-2^{\frac{p+1}{4}})\cdot (X^{\frac{p-1}{2}}+2^{\frac{p+1}{4}})\in \mathbb{F}_p[X]$. In particular, all but one of the roots of $f$ have degree $2$ over the maximal unramified extension $\mathbb{Q}_p^{ur}$ of $\mathbb{Q}_p$, meaning that the inertia group at $p$ acts as an involution with exactly one fixed point. Moreover, the above factorization shows that $\textrm{Gal}(f_t/\mathbb{Q}_p(\sqrt{-2p}))$ is permutation-isomorphic to the Galois group of $X^{p-1} - 2^{\frac{p+1}{2}} = (X^{\frac{p-1}{2}}-2^{\frac{p+1}{4}})\cdot (X^{\frac{p-1}{2}}+2^{\frac{p+1}{4}})$ over $\mathbb{F}_p$, which is an odd order subgroup of $\textrm{Gal}(X^{p-1} - T/\mathbb{F}_p(T)) \cong C_{p-1}$, showing that $p$ has odd residue degree in $K$.

 Finally, to see that the Galois group of the extension $K/\mathbb{Q}$ is $S_n$ for infinitely many $t$ as above, it suffices to note that, so far, we only required $t$ to lie in certain residue classes, i.e., we have a full arithmetic progression of admissible values $t$. It then follows from Hilbert's irreducibility theorem that infinitely many (in fact, ``most", in a density sense) of those $t$ preserve the Galois group $\textrm{Gal}(E(T)/\mathbb{Q}(T))$. Since $E(T)/\mathbb{Q}(T)$ arose from the $S_n$-extension $E/\mathbb{Q}(S)$ via pullback along the solvable extension $\mathbb{Q}(T)/\mathbb{Q}(S)$, its Galois group $\textrm{Gal}(E(T)/\mathbb{Q}(T))$ must be a normal subgroup of $S_n$ with solvable quotient (i.e., one of $A_n, S_n$), and since finally this Galois group contains a transposition (as an inertia group generator, see above), it must be $S_n$. This concludes the proof.
   \end{proof}

Theorem \ref{thm:tech1} can be complemented, for even integers $n$, by the following. 
\begin{theorem}
\label{thm:tech2}
Let $n\ge 6$ be an integer with $n\equiv 2 \bmod 4$. Let $t\in \mathbb{Z}$ be coprime to  $n(n-1)$. Let $f(X):=f_{t}(X)=X^{n-1}(X-1) 
+(n-1)^{-\frac{n}{2}} (\frac{n}{2})^{n-1} t^{-n} \in \mathbb{Q}[X]$.
Then for infinitely many finitely many such $t$, the splitting field $K$ of $f$ is an $S_n$ extension of $\mathbb{Q}$ with the following properties.
\begin{itemize}
\item[i)] $K$ is unramified at all prime divisors of $n$.
\item[ii)] For all prime numbers 
which are ramified in $K$ and do not divide $n-1$, the inertia group is generated by a transposition, say $(a,b)\in S_n$, and the decomposition group is contained in $Sym(\{a,b\})\times Alt(\{1,\dots, n\}\setminus\{a,b\})$.
\item[iii)] If moreover $n$ is such that $n-1$ is either a prime number or a square, then at the prime divisors $p$ of $n-1$, $K$ is either unramified or of ramification index $2$ at $p$; and in the latter case every involution in the decomposition group at $p$ has zero or two fixed points (in the action on the roots of $f$).
\end{itemize}
\end{theorem}

\begin{proof}
The proofs of i) and ii) of the present theorem are completely analogous to those of i) and ii) in Theorem \ref{thm:tech1}, and we refrain from repeating them. For iii), we require certain adaptations. First assume that $p=n-1$ is a prime. Then the Newton polygon of $f(X)$ shows instantly that all roots of $f$ have half-integral $p$-adic valuation and therefore lie in a ramified extension of $\mathbb{Q}_p$.
On the other hand, over {\it any} ramified quadratic extension $\mathbb{Q}_p(\sqrt{\alpha p})/\mathbb{Q}_p$ (with $\alpha\in \mathbb{Z}$ coprime to $p$), consider the polynomial $(\alpha p)^{\frac{n}{2}} (tX)^n f_t(\frac{1}{\sqrt{\alpha p}\cdot tX}) = (1-\sqrt{\alpha p}\cdot tX) + \alpha^{\frac{p+1}{2}} (\frac{n}{2})^p  X^{p+1}$, which modulo the maximal ideal of $\mathbb{Q}_p(\sqrt{\alpha p})$ reduces, up to constant factor, to the separable polynomial $\alpha^{-\frac{p+1}{2}} 2^p  + X^{p+1}\in \mathbb{F}_p[X]$. In particular, all roots of $f$ lie in a quadratic extension of $\mathbb{Q}_p^{ur}$, showing that the inertia group at $p$ is generated by a fixed point free involution. This furthermore forces the decomposition group to be abelian, and if this group contains a further involution, then it must have maximal $2$-elementary abelian quotient $C_2\times C_2$, whence the splitting field of $f$ over $\mathbb{Q}_p$ possesses exactly three quadratic subextensions, exactly two of which are ramified.
Any involution in the decomposition group different from the inertia group generator, fixes one of these two fields elementwise. 
Now, for a given such involution $\sigma$, choose $\alpha$ such that $\mathbb{Q}_p(\sqrt{\alpha p})$ is contained in the fixed field of $\sigma$, and consider again
 the polynomial $\alpha^{-\frac{p+1}{2}} 2^p  + X^{p+1}$ (shown above to have Galois group permutation-equivalent to the one of $f$ over $\mathbb{Q}_p(\sqrt{\alpha p})$), this time over the residue field $\mathbb{F}_{p^d}$ ($d\ge 1$) of $\sigma$. 
Here, if $d$ is even, then this polynomial can have no root (or else it would split completely over $\mathbb{F}_{p^d}$ due to $p+1$-th roots of unity being contained), meaning that $\sigma$ is fixed point free. But if $d$ is odd, then $gcd(p+1,p^d-1)=2$, 
and thus the map $x\mapsto x^{p+1}$ is $2$-to-$1$ on $\mathbb{F}_{p^d}^\times$, implying that this polynomial can have at most two roots in $\mathbb{F}_{p^d}$. Therefore, in all cases, the involutions in the decomposition group have zero or two fixed points.

The analog of the above factorization also shows that all prime divisors of $n-1$ are unramified in the case when $n-1$ is a square.
\end{proof}

\section{Proof of Theorem \ref{thm:main}}
\begin{proof}
Letting $K/\mathbb{Q}$ be an $S_n$ extension as in Theorems
 \ref{thm:tech1} and \ref{thm:tech2} (for the two cases $n\equiv 3\bmod 8$ and $n\equiv 2\bmod 8$ respectively\footnote{Note that the mod-$8$ condition is needed to ensure that the primes in Assertion iii) of the two theorems fall into Condition ii) of Proposition \ref{prop:embed}.}) and $F$ the quadratic subfield of $K/\mathbb{Q}$, it follows from Proposition \ref{prop:embed} that $K/F$ embeds into a $2.A_n$ extension $L/F$ (namely, such that $L/\mathbb{Q}$ is Galois with Galois group $2.S_n^+$) unramified at all non-archimedean primes. To have $L/F$ unramified everywhere, it thus suffices to verify that $F$ can additionally be chosen imaginary-quadratic. 
  This is the case since for $t$ of sufficiently large absolute value, the polynomial $f$ of Theorem \ref{thm:tech1} (resp., of Theorem \ref{thm:tech2})
  has exactly one real root, 
  (resp., has no real root), 
  meaning that in both cases, 
  complex conjugation acts on the roots of $f$ as an involution consisting of $k\equiv 1$ (mod $4$) transpositions; since this is an odd permutation, it follows that $F$ is imaginary quadratic for all such $t$.
  
  Finally, to see that varying the value of $t$ in Theorems \ref{thm:tech1} and \ref{thm:tech2} 
does indeed produce infinitely many {\it distinct} quadratic fields $K$, one may proceed as follows: Recall from the proof of Theorems \ref{thm:tech1} and \ref{thm:tech2} that the function field extensions $\Omega/\mathbb{Q}(T)$ arising as the splitting fields of $f_T(X)$ (where the rational number $t$ is replaced by a transcendental $T$) are $S_n$-extensions of $\mathbb{Q}(T)$; they moreover arise from pulling back the splitting field $E/\mathbb{Q}(S)$ of $X^{n-1}(X-1)+S$ over $\mathbb{Q}(S)$ along the extension $\mathbb{Q}(T)/\mathbb{Q}(S)$, where 
$S=n^{\frac{n-1}{2}}\cdot (\frac{n-1}{2})^{-n}\cdot T^{n-1}$ (in Theorem \ref{thm:tech1}), resp., $S=(n-1)^{-\frac{n}{2}} (\frac{n}{2})^{n-1} T^{-n}$ (in Theorem \ref{thm:tech1}). Since $E/\mathbb{Q}(S)$ possesses some finite nonzero branch point $S\mapsto \lambda$ at which the inertia group is generated by a transposition, and $\mathbb{Q}(T)/\mathbb{Q}(S)$ is unramified at any finite nonzero point, Abhyankar's lemma (see \cite[Theorem 3.9.1]{St}) implies that $\Omega/\mathbb{Q}(T)$ has inertia group generated by a transposition at each of the $n-1$ (resp., $n$) preimages of $S\mapsto \lambda$. The quadratic subextension of $\Omega/\mathbb{Q}(T)$ is therefore ramified at least at $n-1$ points, and therefore is of genus at least $\frac{n-3}{2}>1$. It then follows directly from Faltings' theorem that, given any number field $F$, there are only finitely many rational values $T\mapsto t\in \mathbb{Q}$ for which this quadratic extension specializes into $F$. Since our set of admissible specialization values $t$ was infinite, this implies that one obtains infinitely many distinct quadratic fields $F$ occurring as the quadratic subfields of some $K/\mathbb{Q}$ as considered here.
 \end{proof}

\begin{remark} 
While Theorem \ref{thm:main} applies to infinitely many integers $n$, it would be nice to have the result for all integers $n\equiv 2,3\bmod 8$. I am not sure at this point whether this is possible with the approach of Theorems \ref{thm:tech1} and \ref{thm:tech2}. The extra (primeness or squareness) assumptions are needed only to make Proposition \ref{prop:embed}ii) applicable to the primes in Case iii) of the two theorems (namely, the primes dividing $n$, resp., dividing $n-1$). To give an example, if the prime $p$ is replaced by an arbitrary integer $n\equiv 3\bmod 4$ in the definition of $f$ in Theorem \ref{thm:tech1}, and then the local behavior at primes $p|n$ is investigated, most of the argument in the proof of Case iii) carries over, but the cyclicity of $\textrm{Gal}(X^{n-1} - T/\mathbb{F}_p(T))$, used at the very end to ensure odd residue degree, fails. Still, computer calculations show that for many, but not for all $n$, Assertion ii) of Proposition \ref{prop:embed} remains applicable. It thus remains to be seen whether a unified argument is possible. 
\end{remark}

\section{Small cases not covered by Theorem \ref{thm:main}}
The smallest integer $n$ covered by Theorem \ref{thm:main} is 
$n=10$. Among smaller values of $n$, the analogous assertion is known to hold for $n=5$ by an application of the ``pullback" idea in the proof of Theorems \ref{thm:tech1} and \ref{thm:tech2} to a more sophisticated family of polynomials (\cite{KoeSL25}).
The flawed proof of the case $n=7$ in \cite{KoeSL27} cannot yet be repaired by the ideas presented here and has to be declared open at this point. We can, however, present a proof for the case $n=6$, explicitly left as an open problem in \cite[Section 4]{KoeSL25}, using an alternative approach. Due to the isomorphism $2.A_6\cong SL_2(\mathbb{F}_9)$, this yields

\begin{theorem}
\label{thm:2a6}
There exist infinitely many quadratic number fields, including infinitely many real quadratic fields, possessing an unramified Galois extension with Galois group $SL_2(\mathbb{F}_9)$.
\end{theorem}

Note that the occurrence of real quadratic fields is an extra compared to the cases covered by Theorem \ref{thm:main}, for which our method above can only produce unramified extensions of imaginary quadratic fields since the underlying function field extensions do not have  totally real fibers. 
Theorem \ref{thm:2a6} also continues the progress on unramified $SL_2(\mathbb{F}_q)$-extensions of quadratic number fields started with the cases $q=7$ in \cite{KoeSL27} and $q=5$ in \cite{KoeSL25} (note that $SL_2(5)\cong 2.A_5$).

Our proof of Theorem \ref{thm:2a6} relies on the following technical result, which could in principle also be used to attack the problem for groups $2.A_n$ with arbitrary $n$, provided that rational functions with very particular properties are found.
Unlike the proof of Theorem \ref{thm:main}, which relied (via the proof of Theorems \ref{thm:tech1} and \ref{thm:tech2}) on Galois extensions of the function field $\mathbb{Q}(S)$ with {\it few} (namely, three) branch points in order to be able to control the local behavior on the level of function field extensions, we here work with extensions with {\it many} branch points and move the control of the local behavior to the arithmetic side.
\begin{proposition}
\label{prop:ratfct}
Let $n\ge 4$. Assume that $f(X) = \frac{g(X)}{h(X)}\in \mathbb{Q}(X)$ ($g,h\in \mathbb{Z}[X]$ coprime) is a degree-$n$ rational function with $2n-2$ distinct critical values $\gamma_i=\frac{\alpha_i}{\beta_i}$, all lying in $\mathbb{Q}\cup\{\infty\}$ (i.e., $\alpha_i,\beta_i$ are coprime integers and $\beta_i\ge 0$, for each $i=1,\dots, 2n-2$). For each $i=1,\dots, 2n-2$, let $q_i(X) = \frac{g(X)-\gamma_i h(X)}{(X-\eta_i)^2}$, where $\eta_i\in\mathbb{Q}$ is the unique double preimage of $\gamma_i$ under $f$ (so that, in particular, $q_i(X)\in \mathbb{Q}[X]$ is separable). 
Let $\Delta_i\in \mathbb{Z}$ be the squarefree number which is equivalent, up to squares, to the discriminant of $q_i(X)$. Finally, let $\Delta(T,S)\in \mathbb{Z}[T,S]$ be the discriminant of $S g(X)-T h(X)$,\footnote{I.e., $\Delta(T,S)$ equals an integer times the product of all the $\beta_i T-\alpha_i S$, $i=1,\dots, 2n-2$.} and let $\mathcal{S}$ be the set of fixed prime divisors of $\Delta$, i.e., the set of prime numbers dividing all values $\Delta(t,s)$ with $t,s\in \mathbb{Z}$. Assume that there exist coprime integers $t_0, s_0$ in $\mathbb{Z}$ with $\frac{t_0}{s_0}\notin\{\gamma_1,\dots, \gamma_{2n-2}\}$ such that
\begin{itemize}
\item[a)] the splitting field of $g(X)-\frac{t_0}{s_0}h(X)$ is unramified at all primes in $\mathcal{S}$, and moreover
\item[b)] for all $i=1,\dots, n-2$, 
there exists a prime $\ell_{i,0}$, congruent modulo $4\prod_{j=1}^{2n-2}\Delta_j$ to the prime-to-$\mathcal{S}$ part of $|\beta_i t_0-\alpha_i s_0|$, 
 with $\gcd(\ell_{i,0}, 2\prod_{j=1}^{2n-2}\Delta_j)=1$ and 
 with $\left( \frac{\Delta_i}{\ell_{i,0}}\right)=1$.
\end{itemize}
Then there are infinitely many $\tau\in \mathbb{Q}$ such that the splitting field $K/\mathbb{Q}$ of $g(X)-\tau h(X)$ fulfills Assumption i) of Proposition \ref{prop:embed}. 
Furthermore, those $\tau$ lead to infinitely many pairwise linearly disjoint field extensions $K/\mathbb{Q}$.
\end{proposition}

The proof of Proposition \ref{prop:ratfct} makes use of (the two-variable case of) a famous result of Green and Tao (see \cite[Corollary 1.9]{GT}) on simultaneous prime values of affine linear forms. 
\begin{theorem}
\label{thm:gt}
Let $f_1,\dots, f_k : \mathbb{Z}^2\to \mathbb{Z}$ be affine linear forms in two variables, and let $K\subset \mathbb{Z}^2$ be an open convex cone such that the following hold:

\begin{itemize} \item[i)] The product 
$\prod_{i=1}^k f_i$ has no fixed prime divisor, i.e., for each prime $p$ there exists
$(x_0, y_0) \in \mathbb{Z}^2$ such that none of $f_1(x_0, y_0),\dots, f_k(x_0, y_0)$ is divisible by $p$.
\item[ii)] The $f_i$ are pairwise affinely independent (i.e., if $a, b, c \in \mathbb{Z}$ such that $af_i +bf_j +c = 0$,
then $a = b = c = 0$).
\end{itemize}
Then there are infinitely many $(x_0, y_0) \in K\cap \mathbb{Z}^2$ such that $|f_1(x_0, y_0)|,\dots, |f_k(x_0, y_0)|$ are simultaneously prime.
\end{theorem}

Since we will invoke the Green-Tao theorem with added residue class conditions, the following elementary observation about the behavior of fixed prime divisors and affine independence under affine linear transformations of the arguments will be useful.
\begin{lemma}
\label{lem:aux}
Let $f_1, f_2 : \mathbb{Z}^2\to \mathbb{Z}$ be affine linear forms, and let $\mathcal{S}_i$ be the set of fixed prime divisors of $f_i$, $i=1,2$. Let $M,a,b$ be integers, with $M\ne 0$, and let $g_i(X,Y)=f_i(MX+a, MY+b)$, $i=1,2$. Then the following hold:
\begin{itemize}
\item[i)] The set of fixed prime divisors of $g_i$ is contained in the union of $\mathcal{S}_i$ and the set of prime divisors of $M$.
\item[ii)] If $f_1,f_2$ are affinely independent, then so are $g_1,g_2$.
\end{itemize}
\end{lemma}
\begin{proof} 
i) follows from the fact that the tranformations $X\mapsto MX+a$, $Y\mapsto MY+b$ are invertible modulo all primes not dividing $M$, and thus $f_i$ and $g_i$ have the same value sets modulo such primes. 
Regarding ii), note that an affine dependency of $g_1,g_2$ yields an affine linear combination of $f_1,f_2$ which vanishes on a whole sublattice of $\mathbb{Z}^2$ and must therefore be the zero form.
\end{proof}

\begin{proof}[Proof of Proposition \ref{prop:ratfct}]
Let $N$ be a product of sufficiently large powers of all primes in $\mathcal{S}$, and set $M=4N\prod_{j=1}^{2n-2}\Delta_j$. 
We first use Green-Tao's theorem (Theorem \ref{thm:gt}) and Assumption b) to derive the following \\
{\it Claim}: there exist {\it infinitely many} pairs $(t,s)\in \mathbb{Z}^2$ (with $gcd(t,s)=1$)
 such that all of the following hold:
\begin{itemize}
\item[1)]
$t\equiv t_0\bmod M$ and $s\equiv s_0 \bmod M$.
\item[2)] $\beta_i t-\alpha_i s$ and $\beta_i t_0-\alpha_i s_0$ have the same sign, for all $i=1,\dots, 2n-2$.
\item[3)] The prime-to-$\mathcal{S}$ part of $|\beta_i t-\alpha_i s|$ is a prime number $\ell_i:=\ell_i(t,s)$ (of course, depending on $t,s$) 
with $\left( \frac{\Delta_i}{\ell_i}\right)=1$, $i=1,\dots, 2n-2$. 
\end{itemize}

The concrete argument for application of Theorem \ref{thm:gt} goes as follows.
Set $f_i(T,S) = \beta_i (M\cdot T + t_0) - \alpha_i (M\cdot S + s_0)$, $i=1,\dots, 2n-2$, 
so that evaluating $f_i$ at $(t',s')\in \mathbb{Z}^2$ corresponds to evaluating $\beta_i T-\alpha_i S$ at $(t,s)$ with $t\equiv t_0\bmod M$ and $s\equiv s_0\bmod M$ as intended.

By Lemma \ref{lem:aux}i), all fixed prime divisors of $f_i$ are either in $\mathcal{S}$ or divide $4\prod_{j=1}^{2n-2}\Delta_j$. On the other hand, the coprimeness condition in Assumption b) implies that $\gcd(f_i(0,0), 4\prod_{j=1}^{2n-2}\Delta_j)$ is at most divisible by primes in $\mathcal{S}$, 
so that all fixed prime divisors of $f_i$ are indeed in $\mathcal{S}$. Moreover, for any $(t,s)\in \mathbb{Z}^2$, we have $f_i(t,s)\equiv \beta_i t_0 -\alpha_i s_0 \bmod N$. Up to assuming that $N$ was chosen of $p$-adic valuation $v_p(N)> v_p(\beta_i t_0 -\alpha_i s_0)$ for all $p\in \mathcal{S}$, it follows that $\nu_i:=\prod_{p\in \mathcal{S}}p^{v_p(f_i(t,s))}$ is a constant independent of $t,s\in \mathbb{Z}$. In particular, $\tilde{f}_i(T,S):=\frac{f_i(T,S)}{\nu_i}$ is still an integer-valued linear form, whose values are coprime to all $p\in \mathcal{S}$, and therefore without any fixed prime divisors. By Lemma \ref{lem:aux}ii), $\tilde{f}_1,\dots, \tilde{f}_{2n-2}$ are moreover pairwise affinely independent. Therefore Theorem \ref{thm:gt} yields the existence of infinitely many $(t',s')\in \mathbb{Z}^2$ such that all of $\tilde{f}_i(t',s')=:\ell_i$, $i=1,\dots, 2n-2$, are prime. Moreover, the values $(t',s')$ may be chosen from an arbitrarily small open cone $K$ containing $(t_0,s_0)$, 
 so that we may additionally demand $|\frac{t_0}{s_0}-\frac{Mt'+t_0}{Ms'+s_0}|=|\frac{t_0}{s_0}-\frac{t}{s}|$ to be arbitrarily small, thus implying that $\beta_i t-\alpha_i s$  
 and $\beta_i t_0-\alpha_i s_0$ 
   have the same sign. 
  
Consider now the Legendre symbol $\left(\frac{\Delta_i}{\ell_i}\right)$ Since $4\nu_i\Delta_i$ divides $M$, we have $\nu_i\cdot \ell_i = |\beta_i t-\alpha_i s| \equiv |\beta_i t_0-\alpha_i s_0| \bmod 4\nu_i\Delta_i$, and furthermore $|\beta_i t_0-\alpha_i s_0| \equiv \nu_i\cdot \ell_{i,0} \bmod 4\nu_i\Delta_i$ due to Assumption b). 
 Therefore, $\ell_{i}\equiv \ell_{i,0} \bmod 4\Delta_i$. Since the value of $\left(\frac{\Delta_i}{\ell_i}\right)$ only depends on the mod-$4\Delta_i$ residue class of $\ell_i$ by quadratic reciprocity, 
it follows from Assumption b) that $\left(\frac{\Delta_i}{\ell_i}\right)=1$ for all primes $\ell_i$ arising in this way. We have thus shown the claim.

Next, we use the main theorem of \cite{KLN} to derive information about the local behavior of ramified primes in the splitting field $K/\mathbb{Q}$ of $g(X)-\tau h(X)$, with $\tau:=\frac{t}{s}:=\frac{Mt'+t_0}{Ms'+s_0}$ (and $t',s'$ as above).
Since $\frac{t}{s} -\frac{t_0}{s_0}\equiv 0\bmod N$, it follows from Assumption a) and Krasner's lemma\footnote{Here we need the ``sufficiently high" powers in the definition of $N$. Those could be made explicit with some effort, but we don't require this at this point.} 
 that all primes in $\mathcal{S}$ are unramified in $K/\mathbb{Q}$. Thus, the only ramified non-archimedean primes in $K/\mathbb{Q}$ are the $\ell_i(t,s)$, $i=1,\dots, n-2$. Due to the infinitude of admissible values $(t,s)$, we may furthermore assume these $\ell_i$ to be outside of any prescribed finite set of primes. We may then apply the Specialization Inertia Theorem and \cite[Theorem 4.1]{KLN} as in the proof of Theorem \ref{thm:tech1}.
Concretely, the inertia group $I_i$ at $\ell_i$ in $K/\mathbb{Q}$ is conjugate in $S_n$ to the inertia group at some ramified place of the splitting field of  $g(X)-Th(X)$ over $\mathbb{Q}(T)$, i.e., at some critical value $T\mapsto \gamma_i$ of $f(X)$, and is hence generated by a transposition $(a,b)\in S_n$. Moreover, the decomposition group $D_i$ at $\ell_i$ is conjugate to a subgroup of the decomposition group at $T\mapsto \gamma_i$ (in particular, a subgroup of the normalizer $\langle (a,b)\rangle\times Sym(\{1,\dots, n\}\setminus\{a,b\})$ of $\langle(a,b)\rangle$), and such that $D_i$ is generated by $I_i$ together with the Frobenius at $\ell_i$ in the residue extension at $T\mapsto \gamma_i$. The latter extension is the splitting field of the degree-$n-2$ polynomial $q_i(X)$, whence our assumption $\left( \frac{\Delta_i}{\ell_i}\right)=1$ amounts to saying that $D_i\le \langle (a,b)\rangle \times Alt(\{1,\dots, n\}\setminus\{a,b\})$, as required for Assumption i) of Proposition \ref{prop:embed}. Application of this proposition now concludes the proof (note that the infinitude of admissible specialization values obtained in the first part of the proof implies that we may assume the sets of primes ramified in the splitting fields $K/\mathbb{Q}$ thus obtained to be pairwise disjoint, readily yielding the linear disjointness assertion).
\end{proof}

\begin{remark}
\label{rem:effectcomp}
An effective verification of conditions a) and b) for a concretely given $f$ goes as follows: first for each $p\in \mathcal{S}$ individually, try to find a non-empty $p$-adically open set of values $t_p$ such that the splitting field of 
$g(X)-t_p h(X)$ is unramified at $p$. E.g., first find one such value $t_p$ and then add multiples of sufficiently high $p$-powers, so that Krasner's lemma (or a similar argument to exclude ramification) becomes applicable. If this is possible, then Chinese remainder theorem guarantees the existence of some $\frac{t_0}{s_0}\in \mathbb{Q}$ and $N\in \mathbb{N}$ such that 
$g(X)-\tau h(X)$ is unramified at all $p\in \mathcal{S}$ for every $\tau$ of the form $\tau=\frac{a+Nt}{b+Ns}$, $t,s\in \mathbb{Z}$. As in the previous proof, restricting to such $\tau$ corresponds to replacing the linear forms $\lambda_i(T,S):=\beta_i T - \alpha_i S$ by $\tilde{\lambda}_i(T,S) = 
\beta_i (N\cdot T + t_0) - \alpha_i (N\cdot S + s_0)$. 
Next, for a given $(t_0,s_0)\in \mathbb{Z}^2$, let $D_i$ be the prime-to-$\mathcal{S}$ part of $|\tilde{\lambda}_i(t_0,s_0)|$ 
 and assume $D_i$ is coprime to $2\prod_{j=1}^{2n-2}\Delta_j$.\footnote{Note that this part of the condition can always be achieved since the sets of admissible values $(t_0,s_0)$ surject onto $(\mathbb{F}_p)^2$ for every $p\notin \mathcal{S}$, making $\tilde{\lambda}_i$ a  surjection modulo $p$.} To verify whether the mod-$4\prod_{j=1}^{2n-2}\Delta_j$ residue class
  of $D_i$ contains a prime with the conditions of b), we do not actually have to identify such a prime. Instead, since every coprime residue class modulo $4\prod_{j=1}^{2n-2}\Delta_j$ contains infinitely many primes by Dirichlet's prime number theorem, for our fixed value $D_i$ (assumed to fulfill the necessary coprimeness assumptions), we may simply pretend $D_i$ is prime and transform $\left( \frac{\Delta_i}{D_i}\right)$ into a product $\pm \prod_j \left( \frac{D_i}{p_{ij}}\right)$ over the
  odd 
  prime divisors $p_{ij}$ of $\Delta_i$, with the arising sign only depending on the mod-$4\Delta_i$ residue class of $D_i$ by quadratic reciprocity. Under certain ``independence assumptions", existence of suitable $D_i$ can even be derived directly from the Chinese remainder theorem; e.g., if for every $i\in \{1,\dots, 2n-2\}$ there exists at least one prime $q_i\notin\mathcal{S}$ dividing $\Delta_i$, but not $\Delta_j$ for any $j<i$, then the above product $\pm \prod_j \left( \frac{D_i}{p_{ij}}\right)$ can be controlled at will by noting again the mod-$q_i$ 
 surjectivity of $\tilde{\lambda}_i$ and accordingly choosing the right value for $\left( \frac{D_i}{q_i}\right)$ (successively, for all $i=1,\dots, 2n-2$) assuming that all other $\left( \frac{D_i}{p_{ij}}\right)$, $p_{ij}\ne q_i$ have been prescribed.
\end{remark}

\begin{proof}[Proof of Theorem \ref{thm:2a6}]
Pick $f(X)=\frac{g(X)}{h(X)}:=\frac{X^6 + 53X^4 - 5940X^2 + 62208}{X(3X^4-172X^2+1600)}\in \mathbb{Q}[X]$. This has the $10=2\cdot 6-2$ rational critical values 
$\gamma_{1,2}=\pm 7$, $\gamma_{3,4}=\pm 79/8$, $\gamma_{5,6}=\pm 189/22$, $\gamma_{7,8}=\pm 918/59$ and $\gamma_{9,10}=\pm 1733/250$. From this, together with the fact that $g(X)-Th(X)$ is separable modulo all primes $p\ge 5$, it follows that $g(X)-Th(X)$ has homogenized discriminant  
$$\Delta(T,S)=(\pm)2^k3^m (T^2-7^2S^2)(8^2T^2-79^2S^2)(22^2T^2-189^2S^2)(59^2T^2-918^2S^2)(250^2T^2-1733^2S^2).$$
The only fixed prime divisors of this are easily seen to be $2$ and $3$ (e.g., already the two evaluations at $(0,1)$ and $(1,0)$ have no further common divisors).  
Computer calculation moreover shows that the residue extensions at the ten critical values (are $S_4$-extensions of $\mathbb{Q}$ and) have discriminants 
{\footnotesize $$\Delta_{1,2} = 17\cdot 23\cdot 43\cdot 101,$$ 
$$\Delta_{3,4} = 7\cdot 13\cdot 23 \cdot 79\cdot 109\cdot 113\cdot 2683,$$ 
$$\Delta_{5,6} = 11\cdot 13\cdot 23\cdot 29\cdot 43\cdot 67\cdot 113\cdot 2281,$$ 
$$\Delta_{7,8} = 17\cdot 43\cdot 53\cdot 59\cdot 67\cdot 101\cdot 151\cdot 2683,$$ 
$$\Delta_{9,10} = 7\cdot 17\cdot 23\cdot 29\cdot 43\cdot 53\cdot 109\cdot 151\cdot 1733\cdot 2281.$$} 


One furthermore verifies that the splitting field of $g(X)-385h(X)$ is unramified at the two fixed prime divisors $2$ and $3$. Following Remark \ref{rem:effectcomp}, we replace the (ten) linear factors $\lambda_i(T,S)$ of $\Delta$ by $\tilde{\lambda}_i(T,S):=\lambda_i(NT+385, NS+1)$, divided by the maximal constant divisor of that expression (for a suitable $N=2^a3^b$). 
Recall that $N$ needs to be chosen such that all specializations at values $\frac{Nt+385}{Ns+1}$ remain unramified, and moreover all values $\lambda_i(Nt+385, Ns+1)$, $t,s\in \mathbb{Z}$, have the same $2$- and $3$-adic valuations. 
In fact, $N=32\cdot 729$ is sufficient to ensure all this. For the valuation claim, this is very easy to verify; for the claim about ramification, a detailed proof is also elementary, but slightly technical, so that we only sketch the procedure for the prime $3$. For this, factorize $g(X)-(385+729k)h(X)$ (with $k$ transcendental) modulo $3$ to obtain $X(X-1)(X+1)^4$, and then factorize $3^{-4}(g(Y)-(385+729k)h(Y))$ modulo $3$, where $Y:=3X-1$. This parameter change sends the two roots congruent to $0$ and $1$ modulo $3$ to infinity, and yields the factorization $X(X-1)(X+1)^2$ for the ``remaining" quartic. Separability behavior of these two factorizations shows that the inertia group at $3$ has at most one orbit of length $2$ and fixed points otherwise, i.e., is either trivial or generated by a transposition - but the latter cannot be the case, since the exponent of $3$ in the polynomial discriminant is quickly verified to be even.

Now computer calculation confirms that, e.g., evaluating $(t_0,s_0)=(783,17)$ in the $\tilde{\lambda}_i$ yields the desired result, namely a vector of (positive) specialization values $\tilde{\lambda}_i(t_0,s_0)$, $i=1,\dots, 10$, each of which is coprime to all the $2\prod_{j=1}^{10}\Delta_j$ and has the ``right" parity $\prod_{p|\Delta_i} \left(\frac{\tilde{\lambda}_i(t_0,s_0)}{p}\right)$.\footnote{These computations are elementary, but it may help to make them explicit for at least one index $i$: we have $\lambda_1(T,S)=T-7S$ and thus $\tilde{\lambda}_1(T,S)=\frac{1}{2\cdot 3^3} 2\cdot 3^3\cdot (432T - 3024S + 7)$, yielding $\tilde{\lambda}_1(t_0,s_0) = 286855$. A prime $\ell$ in the same mod-$4\Delta_1$ residue class as $286855$ would fulfill $\left(\frac{\Delta_1}{\ell}\right) = 1$ if and only if
$1= \left(\frac{286855}{17}\right)\cdot (-\left(\frac{286855}{23}\right))\cdot (-\left(\frac{286855}{43}\right))\cdot \left(\frac{286855}{101}\right)$. All of the latter Legendre symbols actually evaluate to $-1$, yielding product $+1$.}

 Via Proposition \ref{prop:ratfct} and Remark \ref{rem:effectcomp}, this 
%
%
 shows the existence of infinitely many quadratic number fields with $SL_2(\mathbb{F}_9)$-extensions unramified at all non-archimedean primes. 

To get the strengthening for real quadratic fields, note that the above rational function $f(X)$ has a totally real fiber over $\frac{32\cdot729\cdot 783+385}{32\cdot 729\cdot 17+1}$, the point which corresponds to $(t_0,s_0)=(783,17)$ after our transformations above. Hence, there are real fibers in a sufficiently small neighborhood of this point, corresponding to an open cone containing $(t_0,s_0)=(783,17)$. Theorem \ref{thm:gt} thus guarantees the existence of {\it totally real} $S_6$ extensions of $\mathbb{Q}$ fulfilling all assumptions of Proposition \ref{prop:ratfct}. The $SL_2(\mathbb{F}_9)$-extension $L/F$ obtained by applying Proposition \ref{prop:embed} is not necessarily real any more; however, one may apply a simple ``twisting argument" exactly as in the last step of the proof of \cite[Theorem 3.2]{KoeSL25}, to get an unramified $SL_2(\mathbb{F}_9)$-extension $\tilde{L}/F$ which is real; the interested reader may verify that the only additional requirement to carry out this twisting is the existence of at least one ramified prime $p\equiv 3\bmod 4$ of $L/\mathbb{Q}$. In our case, such a prime exists automatically, since among our forms $\tilde{\lambda}_i(T,S)$, there are some reducing to constant $3$ mod $4$.
\end{proof}

\begin{remark}
\begin{itemize}
\item[a)] Proceeding as above, the degree-$5$ rational function $f(x)=\frac{x^2(11x^3+11x^2-31x-15)}{171x^3-309x^2-143x+33}$ can be used to prove the analog of Theorem \ref{thm:2a6} for $2.A_5\cong SL_2(5)$, providing an alternative proof of \cite[Theorem 1.1]{KoeSL25}.
\item[b)]
I am aware of rational functions $f(x)\in \mathbb{Q}(x)$ of degree $d$ with $2d-2$ $\mathbb{Q}$-rational critical values, as needed in Proposition \ref{prop:ratfct}, for all $d\le 6$, but no larger $d$. In view of the problems investigated here, it would be interesting to know about the existence or non-existence of such a function in degree $7$.
\end{itemize}
\end{remark}

{\bf Acknowledgement}: I am indebted to Peter M\"uller for making me aware of the remarkable degree-$6$ rational function used in the proof of Theorem \ref{thm:2a6}. I also thank the anonymous referee for several helpful suggestions.

\end{document}